\documentclass[11pt]{amsproc}

\usepackage{mathtext}
\usepackage[cp1251]{inputenc}

\usepackage{bm}
\usepackage[dvips]{graphicx}
\usepackage{amsmath}
\usepackage{amssymb}
\usepackage{amsxtra}

\usepackage{epsfig}
\usepackage{epic}
\usepackage{eepic}
\usepackage{graphics}
\usepackage{graphicx}
\usepackage{subfigure}

\usepackage{caption}
\def\N{{{\Bbb N}}}
\def\Z{{{\Bbb Z}}}
\def\T{{{\Bbb T}}}
\def\R{{\Bbb R}}

\def\l{{\lambda }}
\def\a{{\alpha }}
\def\D{{\Delta }}

\def\a{{\alpha}}

\def\d{{\delta}}

\def\s{{\sigma}}
\def\vp{{\varphi}}

\def\g{{\gamma }}
\def\w{{\omega }}

\def\){\right)}
\def\({\left(}
\def\supp{\operatorname{supp}}

\def\sinc{\operatorname{sinc}}

\numberwithin{equation}{section}
\newtheorem{corollary}{Corollary}[section]
\newtheorem{lemma}{Lemma}[section]
\newtheorem{theorem}{Theorem}[section]
\newtheorem{proposition}{Proposition}[section]
\newtheorem{remark}{Remark}[section]

\newtheorem*{assuma}{Assumption A}

\theoremstyle{definition} \newtheorem{example}{Example}[section]

\par

\begin{document}

\title[]{Approximation by Kantorovich-type operators in general Banach spaces}

\author[Yurii
Kolomoitsev]{Yurii
Kolomoitsev$^{\text{a, 1}}$}
\address{Institute for Numerical and Applied Mathematics, G\"ottingen University, Lotzestr. 16-18, 37083 G\"ottingen, Germany}
\email{kolomoitsev@math.uni-goettingen.de}

\thanks{$^\text{a}$Institute for Numerical and Applied Mathematics, G\"ottingen University, Lotzestr. 16-18, 37083 G\"ottingen, Germany}

\thanks{$^1$Supported by the German Research Foundation, project KO 5804/3-1}


\thanks{E-mail address: kolomoitsev@math.uni-goettingen.de}

\date{\today}

\begin{abstract}
This paper investigates the approximation properties of linear Kantorovich-type
sampling operators in the setting of general Banach lattices $X$ on the torus
$\mathbb{T}$ and the real line $\mathbb{R}$. Under a natural assumption on the uniform
boundedness of the Steklov averaging operators, we establish  direct and
inverse approximation estimates, as well as strong converse inequalities.
Our framework does not require the underlying spaces to be translation-invariant, thereby covering a wide variety of classes and extending the estimates for Kantorovich-type operators previously known primarily for Lebesgue spaces $L_p$.
\end{abstract}

\maketitle

\section{Introduction}

The study of direct and inverse theorems for approximation by trigonometric
polynomials and entire functions of exponential type is a fundamental topic
in approximation theory and Fourier analysis. Historically, these results
were established in the space $C$ of uniformly continuous functions and the
Lebesgue spaces $L_p$ with $1\le p<\infty$, using both best polynomial
approximations and various convolution processes (see, e.g.,~\cite{DL, timan, TB, Z}).
Since then, these classic theorems have been extended to more general
approximation methods and wider classes of function spaces, including
non-translation-invariant Banach spaces such as weighted Lebesgue, Lorentz,
Orlicz, and variable exponent spaces (see, e.g.,~\cite{D98, IG06, KT26, KY10, V23, Vo23}).

A significant development in this field was presented in~\cite{V23}, where
direct and inverse theorems for best approximations and convolution operators
were established in general, non-translation-invariant Banach lattices $X$.
The key requirement there was the uniform boundedness of the Steklov averaging
operators. This natural assumption, in contrast to previous studies,
significantly expanded the class of admissible spaces and enabled the use
of convenient and flexible moduli of smoothness based on the Steklov averages.
Recently, this approach was extended in~\cite{K26} to obtain direct, inverse,
and even strong converse inequalities for sampling operators in the same general framework.

Together with classical approximation methods, the study of Kantorovich-type
sampling operators has also attracted significant attention. Diverse approximation
results for these semi-discrete processes have been established across a wide
range of function spaces, including weighted $L_{p}$ spaces~\cite{KS21}, Wiener
classes~\cite{KKS20}, Orlicz spaces~\cite{CV14, CPV23, BVBS07}, and
variable exponent spaces~\cite{D25}.
Beyond their theoretical importance, Kantorovich-type
operators carry substantial practical relevance:
compared to classical tools, they utilize local integral averages instead of exact point values, making them well-suited for noisy data and physical measurements.

Although Kantorovich-type operators have been extensively studied in specific
function spaces, a unified framework for their analysis remains underdeveloped.
The main goal of this paper is to fill this gap by investigating these
approximation processes on the torus $\mathbb{T}$ and the real line $\mathbb{R}$
in the setting of general Banach lattices. In particular, in both cases,
we establish direct and inverse inequalities as well as strong converse
inequalities. First, we obtain the corresponding approximation results in an
abstract framework under natural conditions on the underlying operators.
Then, we demonstrate how these conditions can be effectively verified
for standard classes of Kantorovich-type operators.

We emphasize that, unlike most existing results in this direction, we establish
error estimates in general Banach lattices that are not necessarily
translation-invariant, thereby covering a wide range of spaces not previously
considered for Kantorovich-type operators. Moreover, the strong converse
inequalities established in the present paper hold in their strong form,
whereas the corresponding estimates in~\cite{AD24, D25, KP21} contain certain
additional terms.
To illustrate the applicability of our general results, we mention the following
simple trigonometric Kantorovich-type operator constructed via the standard
symmetric Steklov averages:
\begin{equation*}
  K_{n} f(x) := \frac{1}{2n+1}\sum_{k=0}^{2n} A_{1/n}f\left(\tfrac{2\pi k}{2n+1}\right) D_n\left(x-\tfrac{2\pi k}{2n+1}\right),
\end{equation*}
where $A_{h}f(x)=\frac{1}{h}\int_{- {h}/{2}}^{{h}/{2}} f(x+t)\,dt$ and $D_n$ is the
classical Dirichlet kernel. Assuming that both the Steklov operators and the
Fourier partial sums are uniformly bounded in $X(\mathbb{T})$, we show as a
direct consequence that the following equivalence holds
for all $f\in X(\mathbb{T})$:
\begin{equation*}
  \|f-{K}_{n} f\|_{X(\mathbb{T})} \asymp \|(I-A_{1/n}) f\|_{X(\mathbb{T})}
\end{equation*}
(see Proposition~\ref{cor2+}).

The paper is organized as follows. Section~2 contains necessary notation and auxiliary results; namely, in Subsection~2.1 we recall the definition and basic properties of Banach lattices, while in Subsection~2.2 we introduce the corresponding moduli of smoothness and the errors of best approximation and mention some of their properties. Section~3 is devoted to the main results in the periodic case, where Subsection~3.1 presents the direct and inverse estimates, Subsection~3.2 establishes the strong converse inequalities, and Subsection~3.3 provides concrete examples. Finally, Section~4 focuses on the non-periodic case and follows a similar structure: Subsection~4.1
presents direct and inverse inequalities, Subsection~4.2 establishes strong converse inequalities, and Subsection~4.3 illustrates these results with examples.

Throughout the paper, $c$, $C$, and $C_j$, $j=1,2,\dots$, denote positive constants which may vary from line to line and are independent of the essential parameters, typically $n$, $\s$, and $f$.

\section{Notation and auxiliary results}

\subsection{Banach lattices}

Let $\Omega$ be either the unit circle $\mathbb{T}$, identified with the interval $[-\pi, \pi)$, or the real line $\mathbb{R}$. We denote by $\mu$ the Lebesgue measure on $\mathbb{R}$ and the normalized Lebesgue measure on $\mathbb{T}$.

Let $X = X(\Omega)$ be a Banach function lattice on $\Omega$ with respect to the measure $\mu$, that is, a Banach space of measurable functions on $\Omega$ with the norm $\|\cdot\|_X$ satisfying the following properties:
\begin{itemize}
  \item[$1)$] if $f$ is measurable and $g \in X$ satisfy $|f| \le |g|$, then $f \in X$ and $\|f\|_X \le \|g\|_X$,
  \item[$2)$] whenever $f_n\in X$, with $\sup_n \|f_n\|_X<\infty$, and $0\le f_n \uparrow f$, then $f\in X$ and $\|f_n\|_X \to \|f\|_X$,
  \item[$3)$] if $E \subset \Omega$ is measurable and~$\mu(E) < \infty$, then $\chi_E\in X$ and there exists
  a constant $c_E>0$ such that  $\int_E |f(x)|d\mu(x)\le c_E\|f\|_X$.
\end{itemize}

We denote by $X'$ the associate space of $X$ with the norm
\begin{equation*}
  \|g\|_{X'} := \sup_{\|f\|_X \le 1} \bigg|\int_\Omega g(x) \overline{f(x)} \, d\mu(x)\bigg|.
\end{equation*}
Note that $X'$ is also a Banach lattice on $\Omega$ with the norm $\|\cdot\|_{X'}$ satisfying the above properties 1)--3), see~\cite[Chapter~1.2]{BS88}.

The space $X$ is called \emph{translation-invariant} if for all $t \in \Omega$ and $f \in X$, the function $f(\cdot + t)$ belongs to $X$ and satisfies
$\|f(\cdot + t)\|_X = \|f\|_X$. We say that $f$ belongs to the Sobolev-type space $X^r$, $r \in \mathbb{N}$, if $f^{(r-1)}$ is (locally) absolutely continuous on $\Omega$ and $f^{(r)} \in X$.

As usual, $\mathcal{T}_n$ denotes the set of all trigonometric polynomials of degree at most $n$.
The class of band-limited functions  $\mathcal{B}_X^\s=\mathcal{B}_X^\s(\R)$ is given by
$$
\mathcal{B}_X^\s:=\left\{\varphi  \in X\cap \mathcal{S}'(\R)\,:\,\supp\;\widehat{\varphi} \subset [-\s,\s]\right\},
$$
where $\widehat{\varphi}$ denotes the Fourier transform of $\varphi$ in the sense of tempered distributions.
For $f \in L_1(\Omega)$, we define its Fourier transform (or Fourier coefficients if $\Omega = \mathbb{T}$) by
$$
\widehat{f}(\xi)=\mathcal{F}f(\xi)
:=
\frac1{2\pi}\int_\Omega f(x)\,e^{-i \xi x}\,dx,
\quad
\xi \in \mathbb{R}
\ \ (\xi \in \mathbb{Z} \text{ if } \Omega = \mathbb{T}).
$$
For measurable functions $f$ and $g$ on $\Omega$, their convolution is defined by
$$
(f * g)(x)
:=
\int_\Omega f(t)\,g(x - t)\,d\mu(t),
\quad x \in \Omega,
$$
whenever the integral is well defined.

In what follows, let $\Phi_n$ be a linear subspace of $X$ parameterized
by $n \in \mathbb{N}$ (or $n > 0$). This family $(\Phi_n)_n$ is assumed to be nested, meaning that $\Phi_m \subset \Phi_n$ whenever $n \ge m$, with the convention $\Phi_0 = \{0\}$.
Furthermore, elements of $\Phi_n$ satisfy the Bernstein-type inequality
\begin{equation*}
  \|\varphi^{(r)}\|_X \le (Bn)^r \|\varphi\|_X, \quad r \in \mathbb{N}, \quad \varphi \in \Phi_n,
\end{equation*}
where the positive constant $B = B_X$ depends only on $X$. We note that if the underlying space $X$ satisfies certain natural conditions, these properties are shared by all classical families of approximating subspaces, such as trigonometric polynomials, band-limited functions, and shift-invariant spaces generated by a smooth function.


For a locally integrable function $f$ and a parameter $h > 0$, we define the Steklov averaging operators $A_h$ and $\dot{A}_h$ by
\begin{equation*}
A_h f(x)
:=
\frac{1}{h}
\int_{x - \frac{h}{2}}^{x + \frac{h}{2}} f(t)\,dt,
\quad
\dot{A}_h f(x)
:=
A_h f\!\left(x + \frac{h}{2}\right).
\end{equation*}
We will also need the so-called higher-order Steklov averaging operators defined by
\begin{equation*}
  A_{h,r} f:= \sum_{k=1}^r (-1)^{k+1}\binom{r}{k} A_h^k f = f-(I-A_h)^rf
\end{equation*}
and
\begin{equation*}
   \dot A_{h,r} f:= \sum_{k=1}^r (-1)^{k+1}\binom{r}{k} \dot A_h^k f = f-(I-\dot A_h)^rf
\end{equation*}

The results of the present paper are obtained under the following assumption on the Banach lattice $X$.

\begin{assuma}\label{asumpA}
For all $f \in X$ and  $h > 0$, the Steklov average $A_h f$ belongs to $X$ and
$$
\sup_{h>0}\sup_{\|f\|_X\le 1}\|A_h f\|_X<\infty.
$$
\end{assuma}

We note that if $X$ satisfies Assumption~A, then its associate space $X'$ also enjoys the same property. This follows from the representation  of the norm $\|\cdot\|_{X'}$ and the Fubini--Tonelli theorem.

A function $F : \Omega \to [0, +\infty]$ is called \emph{symmetrically decreasing} if it is even and non-increasing on $\Omega \cap [0,+\infty)$. Given a function $K : \Omega \to \mathbb{C} \cup \{\infty\}$, we denote by $K^*$ a symmetrically decreasing majorant of $K$, that is, a symmetrically decreasing function such that
$$
|K(x)| \le K^*(x)
\quad \text{for all } x \in \Omega.
$$
We denote by $\mathcal{R}$ the class of integrable symmetrically decreasing functions on $\Omega$, and by $\mathcal{R}^*$ the class of all functions admitting an integrable symmetrically decreasing majorant.


\begin{lemma}\label{lemS} {\sc (See~\cite{V23}.)}
Let $X$ be a Banach lattice satisfying Assumption~A. If $f \in X$ and $K \in \mathcal{R}^*$, then the convolution $f*K$ exists and is finite almost everywhere, belongs to $X$, and satisfies the inequality
\begin{equation*}
\|f*K\|_X
\le
C\,\|K^*\|_{L_1(\Omega)}\,\|f\|_X,
\end{equation*}
where the constant $C$ depends only on $X$.
\end{lemma}

As noted in~\cite{V23}, a wide range of Banach lattices satisfy Assumption~A. In particular, it naturally holds in classical translation-invariant spaces (such as Lebesgue, Orlicz, Lorentz, and Morrey spaces) that possess the Fatou property~\cite{BS88, KPS82, PKJF2013}. For non-translation-invariant settings, the uniform boundedness of Steklov averages can be deduced from the boundedness of the Hardy--Littlewood maximal operator, which has been extensively studied for weighted and variable exponent spaces (see, e.g.,~\cite{CF13, HH19, L25}). Alternatively, a direct analysis of Steklov averages significantly expands the class of admissible spaces $X$. For instance, this property is satisfied in weighted Lebesgue spaces $L_{p,w}$ with $1 \le p < \infty$ under the Muckenhoupt $\mathrm{A}_p$ condition~\cite{St93}, as well as in periodic unweighted and weighted variable exponent spaces recently investigated in~\cite{V24, V25}.

\subsection{Moduli of smoothness and errors of best approximation}

If $X$ is a translation invariant Banach lattice on $\Omega$, then the integral  (classical) modulus of smoothness of a function $f\in X$
of order $r\in \N$ and step $\d>0$ is defined by
\begin{equation*}
    \w_r(f,\d)_X:=\sup_{0<h\le \d} \Vert \D_h^r f\Vert_X,
\end{equation*}
where
$$
\D_h^r f(x):=\sum_{\nu=0}^r\binom{r}{\nu}(-1)^{\nu} f(x+(r-\nu)h),
$$
$\binom{r}{\nu}=\frac{r (r-1)\dots (r-\nu+1)}{\nu!},\quad \binom{r}{0}=1$.

By virtue of~\cite[Theorem~4.6 and Corollary~3.1]{V23} and~\cite[Theorem~2.4, Ch.~6]{DL},
the following relations hold for any $f\in L_p(\Omega)$ with $1\le p<\infty$, $r\in \mathbb{N}$, and $h>0$:
\begin{equation}\label{modav1}
  \sup_{0\le t\le h}\|(I-A_t)^r f\|_{L_p} \asymp \|(I-A_h)^r f\|_{L_p} \asymp \omega_{2r}(f,h)_{L_p},
\end{equation}
\begin{equation}\label{modav2}
  \sup_{0\le t\le h}\|(I-\dot A_{t})^r f\|_{L_p} \asymp \|(I-\dot A_{h})^r f\|_{L_p} \asymp \omega_r(f,h)_{L_p},
\end{equation}
where $\asymp$ denotes a two-sided inequality with positive constants independent of $f$ and $h$. These equivalences demonstrate that the Steklov-based quantities in~\eqref{modav1} and~\eqref{modav2} replicate the essential behavior of classical moduli of smoothness in $L_p(\Omega)$ (see also Lemmas~\ref{lemm1}--\ref{lemm3}).
Since the classical modulus of smoothness is no longer appropriate and may even lose its meaning for non-translation-invariant spaces $X$, the measures of smoothness constructed via the averaging operators $A_{h}$ and ${\dot A}_{h}$ serve as an effective substitute, enabling us to investigate our problems within a more general framework of Banach lattices.

The error of best approximation of $f \in X(\mathbb{T})$ by trigonometric polynomials from $\mathcal{T}_n$ is defined as
$$
E_n(f)_X := \inf \{ \|f - T_n\|_X : T_n \in \mathcal{T}_n \}.
$$
Similarly, the error of best approximation of $f \in X(\mathbb{R})$ by
band-limited functions of order $\sigma > 0$ is given by
$$
E_\sigma(f)_X := \inf \{ \|f - g_\sigma\|_X : g_\sigma \in \mathcal{B}_X^\sigma \}.
$$
For the properties stated in the subsequent three lemmas, we refer the reader to~\cite{V23}.

\begin{lemma}\label{lemm1}
Let $X$ be a Banach lattice satisfying Assumption~A.
For $f\in X$,\, $h,\l>0$, and $r,s\in \N$, $s\le r$, the following properties hold:

\begin{enumerate}
  \item[(i)]  $\|(I-A_{h})^r f\|_X\le C\|f\|_X$,\\
                $\|(I-\dot A_{h})^r f\|_X\le C\|f\|_X$;
  \medskip


  \item[(ii)]
    $
    C^{-1}\|(I-A_{h})^r f\|_{X}\le \|(I-\dot A_{h})^{2r} f\|_{X} \le C\|(I-A_{h})^r f\|_{X};
    $

  \medskip

  \item[(iii)]  if $f\in X^{2s}$, then  $\|(I-A_{h})^r f\|_{X} \le Ch^{2s}  \|(I-A_{h})^{r-s} f^{(2s)}\|_{X}$,\\
                  if $f\in X^s$, then $\|(I-\dot A_{h})^r f\|_{X} \le Ch^{s}  \|(I-\dot A_{h})^{r-s} f^{(s)}\|_{X}$.
\end{enumerate}
In these inequalities, the constant $C$ is independent of $f$ and $h$.
\end{lemma}

The next lemma gives an analogue of the direct approximation theorem in $X$  for trigonometric polynomials and band-limited functions.

\begin{lemma}\label{lemm2}
Let $X$ be a Banach lattice satisfying Assumption~A.
For $f\in X$, $\g>0$, and $r\in \N$, we have
$$
E_n(f)_X\le C\|(I-\dot A_{\g/n})^r f\|_{X},
$$
where the constant $C$ is independent of $f$ and $n$.
\end{lemma}

The corresponding inverse estimate is established in a more general setting. For any $n \in \mathbb{N}$ (or $n > 0$), we define the error of best approximation with respect to the subspace $\Phi_n$ as
$$
E(f,\Phi_n)_X := \inf_{\varphi \in \Phi_n} \|f - \varphi\|_X.
$$

\begin{lemma}\label{leminvk} (See~\cite[inequality (3.23)]{K26}.)
Let $X$ be a Banach lattice satisfying Assumption~A, and let $s \in \mathbb{N}$.

\begin{equation*}
  \|(I - \dot{A}_{\gamma/n})^s f\|_{X} \le C n^{-s} \sum_{\nu=0}^{n} (\nu+1)^{s-1} E(f,\Phi_\nu)_X,
\end{equation*}
where the constant $C$ is independent of $f$ and $n$.
\end{lemma}

Analogues of the Bernstein and Nikolskii-Stechkin-Boas inequalities are given in the following lemma.

\begin{lemma}\label{lemm3}
Let $X$ be a Banach lattice satisfying Assumption~A.
For $T\in \mathcal{T}_n$ (or $T\in \mathcal{B}_X^n$ in the case $\Omega=\R$), $\g>0$, and $r\in \N$, the following properties hold:
\begin{enumerate}

   \item[(i)]  $\|T^{(r)}\|_X\le C n^r \|T\|_X$;

  \item[(ii)]  if $f\in X$ and $T$ satisfies $\|f-T\|_X\le 2E_n(f)_X$. Then
   $$
   \|T^{(r)}\|_X\le C n^r \|(I-\dot A_{{\g}/{n}})^r f\|_{X}.
   $$
\end{enumerate}
In these inequalities, the constant  $C$ is independent of $f$ and $n$.
\end{lemma}

\section{Main results in the periodic case}\label{main}

In this section, we set $X=X(\T)$ and assume that $\vp_{k,n} \in \Phi_n$, $k=1,\dots, m_n$, are $2\pi$-periodic and $\mathcal{X}_n=(x_{k,n})_{k=1}^{m_n}\subset \T$ is a set of points for which there exist positive constants $\g$ and $\g'$ satisfying
\begin{equation}\label{g}
  \frac{\g}{n}\le x_{k+1,n}-x_{k,n}\le \frac{\g'}{n},\quad k=1,\dots,m_n,
\end{equation}
with the convention $x_{m_n+1,n}=2\pi+x_{1,n}$.

We consider the following sequence of linear Kantotovich-type operators with respect the the averaging operators $A_h$:
\begin{equation*}
  K_{n,r}f(x):=\sum_{k=1}^{m_n} A_{\g/n,r}f(x_{k,n})\vp_{k,n}(x),\quad n\in \N.
\end{equation*}
For convenience, we set ${K}_{0,r}:=0$.

To formulate the main results, we impose the following assumptions on
$K_{n,r}$, involving a given parameter $s\in \N$:
\begin{equation}\label{A1}\tag{$A_1$}
  \|K_{n,r}f\|_X\le C_1\|f\|_{X},\quad f\in X,\quad n\in \N,
\end{equation}

\begin{equation}\label{A2}\tag{$A_2$}
\|T-K_{n,r}T\|_X\le C_2n^{-s}\|T^{(s)}\|_X, \quad T\in \mathcal{T}_n,\quad n\in \N.
\end{equation}

\smallskip

\noindent Here $C_1>0$ depends only on $r$ and $X$;  $C_2>0$ depends only on $s$, $r$, and $X$. 

\smallskip

For specific conditions on $K_{n,r}$ that ensure the validity of~\eqref{A1} and~\eqref{A2}, see Propositions~\ref{prex1} and~\ref{prex2} below.


\subsection{Direct and inverse estimates}\label{secdirinv}

\begin{theorem}\label{thmain}
Let $X$ be a Banach lattice satisfying Assumption~A, and let $r, s \in \N$.
Suppose that $K_{n,r}$ satisfies assumptions~\eqref{A1} and~\eqref{A2}. Then for all  $f\in X$, we have
\begin{equation*}
  \|f-K_{n,r} f\|_X\le C\|(I-\dot A_{\g/n})^s f\|_{X},
\end{equation*}
 where the constant $C$ is independent of $f$ and $n$.
\end{theorem}

\begin{proof}
  Let $T\in \mathcal{T}_n$ be such that $\|f-T\|_X\le 2E_n(f)_X$. Applying assumptions~\eqref{A1} and~\eqref{A2} as well as Lemmas~\ref{lemm2} and~\ref{lemm3}(ii), we obtain
   \begin{equation*}
      \begin{split}
         \|f-K_{n,r}f\|_X&\le \|f-T\|_X+\|T-K_{n,r}T\|_X+\|K_{n,r}(f-T)\|_X\\
         &\le C(\|f-T\|_X+n^{-s}\|T^{(s)}\|_X)\\
         &\le C\|(I-\dot A_{\g/n})^sf\|_X.
      \end{split}
   \end{equation*}
\end{proof}

\begin{theorem}\label{thinvk}
Let $X$ be a Banach lattice satisfying Assumption~A, and let $r,s \in \N$.
For all $f\in X$, we have
\begin{equation*}
\begin{split}
  \|(I-\dot A_{\g/n})^s f\|_{X}\le  \frac{C}{n^s}\sum_{k=0}^{n} (\nu+1)^{s-1}\|f-K_{\nu,r} f\|_X,
\end{split}
\end{equation*}
where the constant $C$ is independent of $f$ and $n$.
\end{theorem}

\begin{proof}
  The proof follows directly from Lemma~\ref{leminvk} and the fact that $K_{\nu,r} f\in \Phi_\nu$ for each  $\nu=0,\dots,n$.
\end{proof}

The application of Theorems~\ref{thmain} and~\ref{thinvk} yields the following corollary.

\begin{corollary}\label{cor2}
  Under the assumptions of Theorem~\ref{thmain}, the following properties are equivalent for all  $\a \in (0,s)$:
  \begin{itemize}
    \item[(i)]  $\|f-K_{n,r} f\|_X=\mathcal{O}(n^{-\a})$,
    \item[(ii)] $\|(I-\dot A_{\g/n})^s f\|_{X}=\mathcal{O}(n^{-\a})$.
  \end{itemize}
\end{corollary}

The operators $K_{n,r}$ can be expressed in terms of the general sampling operators
\begin{equation*}
  G_{n}f(x):=\sum_{k=1}^{m_n} f(x_{k,n})\vp_{k,n}(x)
\end{equation*}
as follows:
$$
K_{n,r}f= G_{n}(A_{\g/n,r}f).
$$

Assuming that the function $f$ is continuous on $\T$, we introduce a discrete seminorm
$$
\|f\|_{X_n}:=\bigg\|\sum_{k=1}^{m_n}|f(x_{k,n})|\chi_{[x_{k,n},x_{k+1,n})}\bigg\|_{X}
$$
and impose the following two conditions on $G_n$ (which serve as counterparts
to~\eqref{A1} and~\eqref{A2}):
\begin{equation}\label{a1}\tag{$B_1$}
  \|G_nf\|_X\le C_1\|f\|_{X_n},\quad f\in C(\T),\quad n\in \N,
\end{equation}

\begin{equation}\label{a3}\tag{$B_2$}
\|T-G_nT\|_X\le C_2n^{-s}\|T^{(s)}\|_X, \quad T\in \mathcal{T}_n,\quad n\in \N.
\end{equation}

\smallskip

\noindent Here $C_1>0$ depends only on $X$ and $C_2>0$ depends only on $s$ and $X$.

\smallskip

\begin{remark}\label{remGn}
We note that if $G_n$ satisfies~\eqref{a1} and~\eqref{a3}, then the conclusions of Theorems~\ref{thmain} and~\ref{thinvk} remain valid for $K_{n,r}f= G_{n}(A_{\g/n,r}f)$.
Indeed, by~\eqref{a1} and Lemma~3.1 from~\cite{K26}, we have
\begin{equation*}
  \|K_{n,r}f\|_X=\|G_{n}A_{\g/n,r}f\|_X\le C \|A_{\g/n,r}f\|_{X_n}\le C\|f\|_{X}.
\end{equation*}
At the same time, by~\eqref{a3}, Lemma~\ref{lemm1}(ii), (i), and~(iii), we get
\begin{equation*}
\begin{split}
  \|T-K_{n,r}T\|_X&=\|T-G_n A_{\g/n,r}T\|_X\\
  &\le \|T-A_{\g/n,r}T\|_X+\|A_{\g/n,r}T-G_n A_{\g/n,r}T\|_X\\
  &\le C\(\|(I-\dot A_{\g/n})^{2r} T\|_X+n^{-s}\|(A_{\g/n,r}T)^{(s)}\|_X\)\\
  &\le C\(\|(I-\dot A_{\g/n})^{s} T\|_X+n^{-s}\|A_{\g/n,r}(T^{(s)})\|_X\)\\
  &\le Cn^{-s}\|T^{(s)}\|_X.
\end{split}
\end{equation*}
\end{remark}

\subsection{Strong converse inequalities}\label{secs}
For an approximation method $\mathcal{M}_{n}$, a strong converse inequality of the form
\begin{equation*}
w(f,1/n)_X \le C \|f - \mathcal{M}_n f\|_X,
\end{equation*}
where $w(f,1/n)_X$ is an appropriate measure of smoothness,
is widely used to characterize the rate of convergence (see, e.g.,~\cite{DI93, RS08, TB}).
Together with the corresponding direct estimate, this converse inequality yields the equivalence
\begin{equation*}
\|f - \mathcal{M}_n f\|_X \asymp w(f,1/n)_X.
\end{equation*}
Such equivalence relations have long been known for various convolution-type
operators~\cite{DI93, K12, RS08, T68, T80}, as well as for approximation methods
of a different nature, including Kantorovich-type operators and more general quasi-projection operators (see, e.g.,~\cite{AD24, D25, DI93, KP21, KS21, KS23}).
Most of the existing results concerning Kantorovich-type operators have been
established in classical $L_p$ spaces or, more recently, in variable exponent
spaces (see, e.g.,~\cite{D25}). In this section, we show how to obtain strong converse inequalities for these operators within the framework of general Banach lattices.

For our purposes,  we need to introduce one additional assumption on $K_{n,r}$, involving a fixed parameter $s\in \mathbb{N}$:

{ 
\begin{equation}\label{A3}\tag{$A_3$}
  C_3n^{-s}\|\vp^{(s)}\|_X \le \|\vp-K_{n,r} \vp\|_X, \quad \vp\in \Phi_n,\quad n\in\N,
\end{equation}

\smallskip

\noindent where  $C_3>0$ depends only on $s$ and  $X$.}

\medskip

For specific conditions on $K_{n,r}$ that ensure the validity of~\eqref{A3}, see Proposition~\ref{prex2pr2} below.

\begin{theorem}\label{thC}
 Let $X$ be a Banach lattice satisfying Assumption~A, and $r, s \in \N$. Assume that $K_{n,r}$ satisfies \eqref{A1}, \eqref{A2}, and \eqref{A3}. Then for all $f\in X$, we have
  \begin{equation}\label{c0}
    \|f-K_{n,r} f\|_{X}\asymp\|(I-\dot A_{\g/n})^s f\|_{X},
  \end{equation}
where $\asymp$ denotes a two-sided inequality with positive constants independent of $f$ and $n$.
\end{theorem}

\begin{proof} By Lemma~\ref{lemm1}(i) and~(iii), we have
\begin{equation}\label{c1}
\begin{split}
   \|(I-\dot A_{\g/n})^s f\|_{X}&\le \|(I-\dot A_{\g/n})^s (f-K_{n,r}f)\|_{X}+\|(I-\dot A_{\g/n})^s K_{n,r}f\|_{X}\\
   &\le C\(\|f-K_{n,r}\|_{X}+n^{-s}\|(K_{n,r}f)^{(s)}\|_X\).
\end{split}
\end{equation}
Applying assumptions~\eqref{A3} and~\eqref{A1}, we obtain
\begin{equation}\label{c2}
  \begin{split}
    n^{-s}\|(K_{n,r}f)^{(s)}\|_X &\le C \|K_{n,r} f-K_{n,r} K_{n,r} f\|_X\\
    &=C \|K_{n,r} (f-K_{n,r} f)\|_X\le C\|f-K_{n,r} f\|_{X}.
  \end{split}
\end{equation}
Thus, combining~\eqref{c1} and~\eqref{c2}, we get the lower estimate in~\eqref{c0}.

The corresponding upper estimate follows directly from Theorem~\ref{thmain}.
\end{proof}

It is worth noting that for certain classes of Kantorovich-type operators with $r=2s$, condition~\eqref{A3} can be omitted. In this case, it is sufficient to assume that a simpler identity holds on the subspace $\Phi_{n}$, which leads to the following refinement of Theorem~\ref{thC}.

\begin{proposition}\label{propdopconv}
    Let $X, r, s$ be as in Theorem~\ref{thC} with $r=2s$. Assume that  $K_{n,r}$ satisfies \eqref{A1}, \eqref{A2}, and
    \begin{equation}\label{dopdopconv}
      K_{n,r}\vp=A_{\g/n,r}\vp\quad\text{for every } \vp\in \Phi_n.
    \end{equation}
     Then for all $f\in X$, we have
  \begin{equation}\label{c0+}
    \|f-K_{n,r} f\|_{X}\asymp\|(I-A_{\g/n})^r f\|_{X},
  \end{equation}
where $\asymp$ denotes a two-sided inequality with positive constants independent of $f$ and $n$.
\end{proposition}

\begin{proof}
  We have
\begin{equation}\label{dopconv1}
  \begin{split}
    \|(I-A_{\g/n})^r f\|_{X}&=\|f-A_{\g/n,r} f\|_{X}\\
    &\le \|f-K_{n,r} f\|_{X}+\|A_{\g/n,r} f-K_{n,r} f\|_{X}.
  \end{split}
\end{equation}
Let $\vp\in \Phi_n$ be such that $\|f-\vp\|_X\le 2E(f,\Phi_n)_X$. Then, by Lemma~\ref{lemm1}(i) and~\eqref{A1}, we obtain
\begin{equation}\label{dopconv2}
  \begin{split}
    \|A_{\g/n,r} f-K_{n,r} f\|_{X}&\le \|A_{\g/n,r} f- A_{\g/n,r} \vp\|_{X}+\|K_{n,r} \vp-K_{n,r} f\|_{X}\\
    &\le C\|f-\vp\|_X\le C\|f-K_{n,r} f\|_{X}.
  \end{split}
\end{equation}
Combining~\eqref{dopconv1} and~\eqref{dopconv2}, we get the lower estimate in~\eqref{c0+}.

The upper estimate follows from Theorem~\ref{thmain} and Lemma~\ref{lemm1}(ii).
\end{proof}

\begin{remark}\label{rem1conv}
{\rm (i)} In the classical case $X=L_p(\T)$, the two-sided estimates~\eqref{c0} and~\eqref{c0+} take the following form
$$
\|f-K_{n,r} f\|_{L_p(\T)}\asymp\omega_s(f,1/n)_{L_p(\T)}.
$$

{\rm (ii)} Condition~\eqref{dopdopconv} holds, for example, when $\Phi_n=\mathcal{T}_n$ and $K_{n,r}f = G_{n}(A_{\gamma/n,r}f)$, where the operator $G_n:C(\T)\to \mathcal{T}_n$ satisfies $G_n T=T$ for all $T\in \mathcal{T}_n$.
\end{remark}

\begin{remark}
Theorems~\ref{thmain}, \ref{thinvk}, and~\ref{thC}, as well as Proposition~\ref{propdopconv},
remain true for the Kantorovich-type operators
\begin{equation*}
  \dot K_{n,r}f(x) := \sum_{k=1}^{m_n} \dot A_{\gamma/n,r}f(x_{k,n})\varphi_{k,n}(x), \quad n\in \mathbb{N}.
\end{equation*}
The corresponding proofs carry over directly with only minor adjustments.
\end{remark}

\subsection{Examples}\label{secext}

In this section, we investigate certain classes of Kantorovich-type operators
and provide easily verifiable conditions that ensure the fulfillment
of assumptions~\eqref{A1}, \eqref{A2}, and~\eqref{A3}.

For a given function $\phi$, we denote
\begin{equation*}
  V_n^\phi(x):=\sum_{k\in \Z} \phi\(\frac kn\)e^{ikx}.
\end{equation*}
To illustrate  the application of the results from Sections~\ref{secdirinv} and~\ref{secs}, we consider the following trigonometric Kantorovich-type operator
$$
K_{n,r}^\vp f(x):=\frac{1}{2n+1}\sum_{k=0}^{2n}A_{\g/n,r}f(t_k)V_n^\vp(x-t_k),\quad t_k=\frac{2\pi k}{2n+1},
$$
where  $\vp$ is an even function of bounded variation with $\supp\vp\subset [-1,1]$.

To simplify the notation, we denote
$$
K_n^\varphi f := K_{n,1}^\varphi f.
$$
For our purposes, we employ the general sampling operators
$$
G_{n}^\varphi f(x) := \frac{1}{2n+1}\sum_{k=0}^{2n} f(t_k) V_n^\varphi(x-t_k)
$$
together with the identity
\begin{equation}\label{eqKG}
K_{n,r}^\varphi f(x) = G_n^\varphi (A_{\gamma/n,r}f)(x).
\end{equation}


\begin{proposition}\label{prex1}
   Let $X$ be a Banach lattice such that
   \begin{equation}\label{zv1}
     \sup_{n\in \N}\sup_{\|f\|_X\le 1}\|f*V_n^{\vp}\|_X<\infty.
   \end{equation}
Then the operators $K_{n,r}^\vp$ ($r\in \N$) and $G_n^\vp$ satisfy~\eqref{A1} and~\eqref{a1}, respectively.
In particular, \eqref{zv1} holds if $X$ satisfies Assumption A
and $\widehat{\vp}\in \mathcal{R}^*$.
\end{proposition}

\begin{proof}
The statement of the proposition for the operators $G_n^\vp$ was proved in~\cite[Proposition~3.3]{K26}.
For the Kantorovich operators $K_{n,r}^\vp$, the proof follows by the same arguments as in Remark~\ref{remGn}.
\end{proof}

\begin{lemma}\label{levpn}
  Let $\vp$ be an even function of bounded variation with $\supp\vp\subset [-1,1]$ and $\widehat{\vp}\in \mathcal{R}^*$. Then
  \begin{equation*}
    V_n^\vp\in \mathcal{R}^*\quad\text{with}\quad \sup_n \|(V_n^\vp)^*\|_{L_1(\T)}<\infty.
  \end{equation*}
\end{lemma}

\begin{proof}
  The proof follows the same lines as the second part of the proof of Proposition~3.3 in~\cite{K26}.
\end{proof}

The following remark provides simple conditions on $\vp\in L_1(\R)$ to ensure that $\widehat{\vp}\in \mathcal{R}^*$.

\begin{remark}\label{reml1}
For any $\vp\in L_1(\R)$ and $r\in \N$ with $r\ge 2$, the following inequality holds (see, e.g.,~\cite{Treb77}):
$$
|\widehat{\vp}(\xi)|\le C\w_r\(\vp,\frac1{|\xi|}\)_{L_1(\R)},
$$
where the constant $C$ is independent of $\vp$ and $\xi$.
In particular, if
$$
\displaystyle\int_0^1\frac{\w_r(\vp,t)_{L_1(\R)}}{t^2}dt<\infty
$$
or, more stronger,  $\vp\in W_1^2(\R)$, then $\widehat{\vp}\in \mathcal{R}^*$.   This is a consequence of the fact that the modulus of smoothness is an increasing function and $\w_r(\vp,h)_{L_1(\R)}\le Ch^2\|\vp''\|_{L_1(\R)}$ (see, e.g.,~\cite{KT20}).
\end{remark}

To formulate the next propositions, we denote by $\eta$ a function in $C^\infty(\R)\cap \mathcal{R}$ such that $\eta(\xi)=1$ for $|\xi|\le 1$ and $\eta(\xi)=0$ for $|\xi|\ge 2$.

\begin{proposition}\label{prex2}
Let $X$ be a Banach lattice satisfying Assumption A and let
$$
w(\xi)=\frac{1-\varphi(\xi)}{\xi^s}\eta(\xi),
$$
where $w$ is understood to be continuously extended to $\xi=0$.
If $w$ is of bounded variation on $\R$ and $\widehat{w}\in\mathcal R^*$, then the operators $K_{n,r}^\vp$ ($r\in \N$ with $2r\ge s$) and $G_n^\vp$ satisfy~\eqref{A2} and~\eqref{a3}, respectively.
\end{proposition}

\begin{proof}
The statement of the proposition for the operators $G_n^\varphi$ was established in~\cite[Proposition~3.4]{K26}. Consequently, the assertion for the Kantorovich-type
operators follows directly from Remark~\ref{remGn}.
\end{proof}

We denote
$$
\sinc \xi:=\left\{
            \begin{array}{ll}
              \displaystyle\tfrac2\xi\sin\tfrac \xi2, & \hbox{$\xi\neq 0$,} \\
              \displaystyle1, & \hbox{$\xi=0$.}
            \end{array}
          \right.
$$

\begin{proposition}\label{prex2pr2}
   Let $X$ be a Banach lattice satisfying Assumption A and let
     \begin{equation*}
     v(\xi)=\frac{\xi^s\eta(\xi)}{1-\vp(\xi)\(1-\(1-\sinc(\g\xi)\)^r\)},
    \end{equation*}
    where  $v$ is understood to be continuously extended to $\xi=0$.
    If $v$ is of bounded variation on $\R$ and $\widehat{v}\in\mathcal R^*$, then the operator $K_{n,r}^\vp$ ($r\in\N$ with $2r\ge s$) satisfies~\eqref{A3}.
\end{proposition}

\begin{proof}
  By the quadrature formula for trigonometric polynomials, for any $T\in \mathcal{T}_n$, we have
   $$
   K_{n,r}^\vp T=A_{\g/n,r} T*V_n^\vp,
   $$
   which implies that
   $$
   T^{(s)}=(T-A_{\g/n,r} T*V_n^\vp)*V_n^v=(T-K_{n,r}^\vp T)*V_n^v.
   $$
  This together with the principle of comparison for Fourier multipliers (see, e.g.,~\cite[Ch.~7]{TB} or \cite[Sec.~2.4]{V23}) implies that inequality~\eqref{A3} holds for $K_{n,r}^\vp$  if and only if there exists a constant $C>0$ such that
  \begin{equation}\label{TUn}
    \|T*V_n^v\|_X\le C\|T\|_X,\quad T\in\mathcal{T}_{n}.
  \end{equation}
  By Lemma~\ref{levpn}, we have that $V_n^v\in\mathcal{R}^*$ and $\sup_n \|(V_n^v)^*\|_{L_1(\T)}<\infty$. Thus, applying Lemma~\ref{lemS}, we obtain~\eqref{TUn}, which proves the proposition.
\end{proof}

\begin{example}\label{ex1}
Consider the Kantorovich operator ${K}_{n}^{\vp_1}$ generated by the function $\vp_1=\chi_{[-1,1]}$ such that $V_n^{\vp_1}(x)=D_n(x)=\sum_{k=-n}^n e^{ikx}$ is the standard Dirichlet kernel.

\begin{proposition}\label{cor2+}
Let $X$ be a Banach lattice satisfying Assumption~A. Suppose that
\begin{equation}\label{fDn}
\sup_{n\in \N}\sup_{\|f\|_X\le 1}\|f*D_n\|_{X}<\infty.
\end{equation}
Then for all $f\in X$, we have
  \begin{equation*}
    \|f-{K}_{n}^{\vp_1} f\|_{X}\asymp\|(I-A_{\g/n}) f\|_{X},
  \end{equation*}
where $\asymp$ denotes a two-sided inequality with positive constants independent of $f$ and $n$.
\end{proposition}

\begin{proof}
By Proposition~\ref{prex1}, we have that ${K}_{n}^{\vp_1}$ satisfies assumption~\eqref{A1}.
Further, applying equality~\eqref{eqKG} and the well-known fact that
$$
G_n^{\vp_1}T=T,\quad T\in \mathcal{T}_n,
$$
we obtain
\begin{equation*}
  \begin{split}
      \|T-{K}_{n}^{\vp_1} T\|_X&=\|T-G_n^{\vp_1}A_{\g/n}T\|_X\\
      &\le \|T-A_{\g/n}T\|_X+\|A_{\g/n}T-G_n^{\vp_1}A_{\g/n}T\|_X\\
      &=\|T-A_{\g/n}T\|_X\le Cn^{-2}\|T''\|_X,
   \end{split}
\end{equation*}
where in the latter inequality we applied Lemma~\ref{lemm1}(iii). Thus, the operator ${K}_{n}^{\vp_1}$ satisfies~\eqref{A2}. It remains only to apply Proposition~\ref{propdopconv} taking into account Remark~\ref{rem1conv}(ii).
\end{proof}
\end{example}

\begin{example}\label{ex2}
Let us denote
$$
\rho(\xi) := (1 - \xi^2)_+^\alpha, \quad \alpha > 0.
$$
Then, $K_{n}^\rho$ represents the Kantorovich-type Bochner--Riesz operator.
The next proposition provides strong converse inequalities for this class of operators.

\begin{proposition}\label{br}
  Let $X$ be a Banach lattice satisfying Assumption A. Then for all $f\in X$ we have
  \begin{equation*}
    \|f-{K}_{n}^\rho f\|_{X}\asymp \|(I-A_{\g/n}) f\|_{X},
  \end{equation*}
where $\asymp$ denotes a two-sided inequality with positive constants independent of $f$ and $n$.
\end{proposition}

\begin{proof}
It is well known that there exists a constant $C>0$ such that
\begin{equation*}
  |\widehat{\rho}(x)|\le \frac{C}{(1+|x|)^{1+\a}}\quad\text{for all }x\in \R.
\end{equation*}
This estimate together with Proposition~\ref{prex1} implies that ${K}_{n}^\rho$ satisfies assumption~\eqref{A1}.

In~\cite[Example~3.2]{K26}, we established that
$$
w(\xi)=\frac{1-\rho(\xi)}{\xi^2}\eta(\xi)\in \mathcal{R}^*.
$$
Hence, by Proposition~\ref{prex2}, the operator ${K}_{n}^\rho$ satisfies assumption~\eqref{A2}.

We now show that ${K}_{n}^\rho$ satisfies also assumption~\eqref{A3}.
To this end, we consider the following functions:
$$
u(\xi)=\frac{\xi^2}{1-\rho(\xi)\sinc\g\xi},
$$
$$
v(\xi)=u(\xi)\eta(\xi),
$$
and
$$
\eta_0(\xi)=\eta(4\xi),\qquad\eta_1(\xi)=\eta(\xi)-\eta_1(\xi).
$$
As above, the function $u$ is understood to be continuously extended to $\xi=0$

It is easy to see that $u\eta_0\in W_1^2(\R)$. Hence, by Remark~\ref{reml1}, we get
\begin{equation}\label{br0}
  \widehat{u\eta_0}\in \mathcal{R}^*.
\end{equation}

To investigate the function $u\eta_1$, we follow~\cite{RS08}.
Denoting
$$
\psi(\xi)=\rho(\xi)\sinc(\g\xi),
$$
we represent $u\eta_1$ as
\begin{equation*}
  u\eta_1(\xi)=\xi^2\(\eta_1(\xi)\sum_{j=0}^m (\psi(\xi))^j+\eta_1(\xi)g(\xi)\),
\end{equation*}
where
$m\ge \frac{2}{\a}$, $m\in\N$, and
$$
g(\xi)=\frac{((1-\xi^2)_+^\a\sinc(\g\xi))^m}{1-(1-\xi^2)_+^\a\sinc(\g\xi)}.
$$
It is easy to see that
\begin{equation*}
  \lim_{\xi\to1+0}g^{(s)}(\xi)=0,\quad s\in \N.
\end{equation*}
We now show that
\begin{equation}\label{br3}
  \lim_{\xi\to1-0}g^{(s)}(\xi)=0,\quad s=1,2.
\end{equation}
To this end, we need the following simple lemma.
\begin{lemma}\label{lebr} (See~\cite[p.~33]{RS08}.)
  Let
$$
F(x)=\frac{x^{\a m}}{1-x^\a},\quad 0\le x<1.
$$
Then
$$
F^{(s)}(x)=x^{\a m-s}h(x),
$$
where $h(x)$ is continuous for $0\le x<1$. In particular,
\begin{equation}\label{br4}
  \lim_{x\to +0}F^{(s)}(x)=0,\quad s=1,2.
\end{equation}
\end{lemma}

Applying this lemma, we obtain
\begin{equation*}
\begin{split}
    g^{(s)}(\xi)&=\(F\((1-\xi^2)(\sinc \g\xi)^{1/\a}\)\)^{(s)}\\
    &=\sum_{j=1}^s F^{(j)}\((1-\xi^2)(\sinc \g\xi)^{1/\a}\)P_{s,j}(\xi),
\end{split}
\end{equation*}
where $P_{s,j}$ are some functions from $C^\infty(\R)$. This, together with~\eqref{br4}, yields~\eqref{br3}. Now, it is straightforward to verify that $\xi^2\eta_1g(\xi)\in W_1^2(\R)$, which, by Remark~\ref{reml1}, implies $\mathcal{F}(\xi^2\eta_1g(\xi))\in \mathcal{R}^*$. It can also be easily verified that $\mathcal{F}(\xi^2\eta_1(\xi))$  and $\widehat{\psi^j}$ belong to $\mathcal{R}^*$; hence, we conclude  that $\widehat{u\eta_1}\in \mathcal{R}^*$.
Combining this with~\eqref{br0}, we find that $\widehat{v}\in \mathcal{R}^*$. Therefore, applying Proposition~\ref{prex2pr2}, we obtain that the operator ${K}_{n}^\rho$ satisfies assumption~\eqref{A3}.

Finally, the assertion of the proposition follows by applying Theorem~\ref{thC} and Lemma~\ref{lemm1}(ii).
\end{proof}

We remark that in the classical case $X=L_p(\T)$, $1\le p<\infty$, Propositions~\ref{br} and~\ref{cor2+} were obtained in~\cite{KP21}.
\end{example}

\begin{example}\label{ex1.5} We now construct Kantorovich operators that provide the same order of convergence as the error of best approximation $E_n(f)_X$.

We set
$$
\varphi_2(\xi) = \frac{\eta(2\xi)}{\operatorname{sinc}(\gamma\xi)} \quad\text{with}\quad 0 < \gamma \le \frac{\pi}{2}.
$$
Note that here $\gamma$ is the same parameter as in~\eqref{g} and in the definition
of the corresponding Kantorovich-type operators.

\begin{proposition}\label{prKEas}
  Let $X$ be a Banach lattice satisfying Assumption~A. Then
  \begin{equation}\label{eqKE}
    E_n(f)_X\le \|f-{K}_{n}^{\vp_2} f\|_X\le CE_{n/2}(f)_X,
  \end{equation}
  where the constat $C$ is independent of $f$ and $n$.
\end{proposition}

\begin{proof}
It is straightforward to verify that $\varphi_2 \in W_1^2(\mathbb{R})$,
which implies that $V_n^{\varphi_2} \in \mathcal{R}^*$ and
$\sup_n \|(V_n^{\varphi_2})^*\|_{L_1(\mathbb{\T})} < \infty$
by virtue of Lemma~\ref{levpn} and Remark~\ref{reml1}. Therefore, by Proposition~\ref{prex1}, the operator ${K}_{n}^{\vp_2}$ satisfies assumption~\eqref{A1}.
Further, let $T\in \mathcal{T}_{n/2}$ be such that $\|f-T\|_X\le 2E_{n/2}(f)_X$. Taking into account that by the quadrature formula,
   $$
   {K}_{n}^{\vp_2} T=A_{\g/n} T*V_n^{\vp_2}=T,
   $$
and applying inequality~\eqref{A1}, we obtain
\begin{equation*}
  \begin{split}
     \|f-{K}_{n}^{\vp_2} f\|_X&\le \|f-T\|_X+\|T-{K}_{n}^{\vp_2} T\|_X+\|{K}_{n}^{\vp_2} (f-T)\|_X\\
     &=\|f-T\|_X+\|{K}_{n}^{\vp_2} (f-T)\|_X\\
     &\le C\|f-T\|_X\le CE_{n/2}(f)_X.
   \end{split}
\end{equation*}
The lower estimate in~\eqref{eqKE} follows from the fact that ${K}_{n}^{\vp_2} f\in \mathcal{T}_n$.
\end{proof}

Under an additional restriction on $X$, the Kantorovich operator ${K}_{n}^{\vp_3}$, where
$$
\vp_3(\xi)=\frac{\chi_{[-1,1]}(\xi)}{\sinc(\g\xi)},\quad 0<\g\le\frac\pi2,
$$
provides more shaper version of~\eqref{eqKE}.

\begin{proposition}\label{prKEas2}
  Let $X$ be a Banach lattice satisfying Assumption~A and~\eqref{fDn}. Then
  \begin{equation*}
    \|f-{K}_{n}^{\vp_3} f\|_X\asymp E_{n}(f)_X,
  \end{equation*}
where $\asymp$ is a two-sided inequality with positive constants independent of $f$ and $n$.
\end{proposition}

\begin{proof}
First, we show that ${K}_{n}^{\vp_3}$ satisfies assumption~\eqref{A1}. To this end, we use the following representation for the kernel $V_{n}^{\vp_3}$,
  $$
  V_n^{\vp_3}=V_n^{\vp_1}*V_{n}^{\vp_3\eta},
  $$
  It is clear that $\vp_3\eta\in W_1^2(\R)$. Hence, applying Lemma~\ref{levpn} together with Remark~\ref{reml1}, we get that
  \begin{equation*}
    V_{n}^{\vp_3\eta}\in \mathcal{R}^*\quad\text{with}\quad \sup_n \|(V_{n}^{\vp_3\eta})^*\|_{L_1(\T)}<\infty.
  \end{equation*}
  This together with the condition~\eqref{fDn} yields
  $$
  \|f*V_n^{\vp_3}\|_X=\|f*V_n^{\vp_1}*V_{n}^{\vp_3\eta}\|_X\le C \|f*V_n^{\vp_1}\|_X \le C\|f\|_X.
  $$
  Thus, using Proposition~\ref{prex1}, we conclude that ${K}_{n}^{\vp_3}$ satisfies assumption~\eqref{A1}.

The rest of the proof is analogous to the proof of Proposition~\ref{prKEas}, using the facts that
$$
{K}_{n}^{\vp_3}T=A_{\g/n}T*V_n^{\vp_3}=T\quad\text{for all } T\in \mathcal{T}_n,
$$
and that ${K}_{n}^{\vp_3}f\in \mathcal{T}_n$.
\end{proof}
\end{example}

\section{Main results in the non-periodic case}
Analogues of the results from Section~\ref{main} remain valid for non-periodic functions $f$  defined on $\R$ and Kantorovich-type operators
\begin{equation*}
\mathcal{K}_{\s,r} f(x):=\sum_{k\in \Z} A_{\g/\s,r}f(x_{k,\s})\vp_{k,\s}(x),
\end{equation*}
where $\vp_{k,\s} \in \Phi_\s$, $k\in \Z$, and  $\mathcal{X}_\s=(x_{k,\s})_{k\in\Z}$ is a set of distinct points for which there exist positive constants $\g$ and $\g'$ satisfying
\begin{equation*}
  \frac{\g}{\s}\le x_{k+1,\s}-x_{k,\s}\le \frac{\g'}{\s},\quad k\in \Z.
\end{equation*}
For convenience, we set $\mathcal{K}_{0,r}:=0$. 

Throughout this section, we assume that the series defining $\mathcal{K}_{\sigma,r} f$
converges in the norm of $X$ for any $f \in X(\mathbb{R})$. In order to guarantee this convergence for Kantorovich-type operators with non-compactly supported kernels
$\varphi_{k,\sigma}$, we additionally require that the Banach lattice $X(\mathbb{R})$ possesses an order continuous norm and contains the Schwartz space $\mathcal{S}$ as a dense
subspace. We recall that the norm in $X$ is order continuous if for any sequence such that $f_n \downarrow 0$ a.e., it follows that $\|f_n\|_X \to 0$ (see, e.g.,~\cite[Ch.~2, Sec.~2.4]{MN91}).

Similarly to the periodic case, we introduce the discrete seminorm
$$
\|f\|_{X_\sigma(\mathbb{R})} := \bigg\| \sum_{k\in \mathbb{Z}} |f(x_{k,\s})| \chi_{[x_{k,\s}, x_{k+1,\s})} \bigg\|_{X(\mathbb{R})}
$$
and define the space $X_\sigma(\mathbb{R})$ as the set of all functions $f \in X(\mathbb{R})$
such that $\|f\|_{X_\sigma(\mathbb{R})} < \infty$. According to~\cite[Lemma~3.1]{K26},
for any $f \in X(\mathbb{R})$, the following inequality holds:
\begin{equation}\label{Asfx}
  \|A_{\gamma/\sigma} f\|_{X_\sigma(\mathbb{R})} \le C \|f\|_{X(\mathbb{R})},
\end{equation}
where the constant $C$ is independent of $f$ and $\sigma$.

In this section, for convenience, we write $X=X(\R)$ and $X_\s=X_\s(\R)$.

We use the following assumptions on
$\mathcal{K}_{\s,r}$, involving a fixed parameter $s\in \N$:

\begin{equation}\label{a1r}\tag{$\mathcal{A}_1$}
  \|\mathcal{K}_{\s,r} f\|_X\le C_1\|f\|_{X},\quad f\in X,\quad \s\ge 1,
\end{equation}

\begin{equation}\label{a3r}\tag{$\mathcal{A}_2$}
  \|g-\mathcal{K}_{\s,r} g\|_X\le C_2\s^{-s}\|g^{(s)}\|_X, \quad g\in \mathcal{B}_X^\s,\quad \s\ge 1,
\end{equation}

\begin{equation}\label{a4r}\tag{$\mathcal{A}_3$}
  C_3\s^{-s}\|\vp^{(s)}\|_X \le \|\vp-\mathcal{K}_{\s,r} \vp\|_X, \quad \vp\in \Phi_\s,\quad \s\ge 1.
\end{equation}

\smallskip

\noindent Here $C_1>0$ depends only on $r$ and $X$;  $C_2, C_3>0$ depend only on $s$, $r$, and $X$. 

\smallskip

The proofs from Section~\ref{main} for the periodic setting carry over
to the spaces $X(\mathbb{R})$ and the operators $\mathcal{K}_{\sigma,r}$ via standard
adjustments. Therefore, we restrict our attention here to formulating the main
theorems and providing a detailed verification of the examples. This is essential
both for justifying the use of convolution operators and for ensuring that the
corresponding Kantorovich-type series are well-defined.

\subsection{Direct and inverse inequalities}\label{secir}

By arguments analogous to those used in the proofs of Theorems~\ref{thmain}
and~\ref{thinvk}, we derive the following direct and inverse estimates
for Kantorovich-type operators on $\mathbb{R}$.

\begin{theorem}\label{thstr}
Let ${X(\R)}$ be a Banach lattice satisfying Assumption~A, and let $r, s \in \N$. Suppose that $\mathcal{K}_{\s,r}$ satisfies assumptions~\eqref{a1r} and~\eqref{a3r}. Then for all $f\in X$, we have
\begin{equation*}
  \|f-\mathcal{K}_{\s,r} f\|_{X}\le C\|(I-\dot A_{\g/\s})^s f\|_{X},
\end{equation*}
where the constant $C$ is independent of $f$ and $\s$.
\end{theorem}

\begin{theorem}\label{thinvr}
Let $X(\R)$ be a Banach lattice satisfying Assumption~A, and let $r, s \in \N$.
Then for all $f\in X$, we have
\begin{equation*}
\begin{split}
  \|(I-\dot A_{\g/\s})^s f\|_{X} \le \frac{C}{\s^s}\sum_{\nu=0}^{\lfloor\s\rfloor} (\nu+1)^{s-1}\|f-\mathcal{K}_{\nu,r} f\|_{X},
\end{split}
\end{equation*}
where the constant $C$ is independent of $f$ and $\s$.
\end{theorem}

The application of Theorems~\ref{thstr} and~\ref{thinvr} yields the following corollary.

\begin{corollary}\label{cor2r}
  Under the conditions of Theorem~\ref{thstr}, the following properties are equivalent for all $\a \in (0,s)$:
  \begin{itemize}
    \item[(i)]  $\|f-\mathcal{K}_{\s,r} f\|_{X}=\mathcal{O}(\s^{-\a})$,
    \item[(ii)] $\|(I-\dot A_{\g/\s})^s f\|_{X}=\mathcal{O}(\s^{-\a})$.
  \end{itemize}
\end{corollary}

In a similar manner to the periodic case, the operators $\mathcal{K}_{\sigma,r}$ can be
expressed in terms of the general sampling operators
\begin{equation*}
  \mathcal{G}_{\sigma}f(x) := \sum_{k\in \mathbb{Z}} f(x_{k,\sigma})\varphi_{k,\sigma}(x)
\end{equation*}
via the relation
$$
\mathcal{K}_{\sigma,r}f = \mathcal{G}_{\sigma}(A_{\gamma/\sigma,r}f).
$$

We impose the following two conditions on $\mathcal{G}_\sigma$:
\begin{equation}\label{b1r}\tag{$\mathcal{B}_1$}
  \|\mathcal{G}_\sigma f\|_X \le C_1 \|f\|_{X_\sigma}, \quad f \in X_\sigma(\mathbb{R}), \quad \sigma \ge 1,
\end{equation}
\begin{equation}\label{b2r}\tag{$\mathcal{B}_2$}
  \|g - \mathcal{G}_\sigma g\|_X \le C_2 \sigma^{-s} \|g^{(s)}\|_X, \quad g \in \mathcal{B}_X^\sigma, \quad \sigma \ge 1.
\end{equation}
Here, the positive constant $C_1$ depends only on $X$, whereas $C_2$ depends only on $s$ and $X$.

\begin{remark}\label{remGnr}
Analogously to Remark~\ref{remGn} dealing with the periodic case, it can be shown
that if $\mathcal{G}_\sigma$ satisfies~\eqref{b1r} and~\eqref{b2r}, then the conclusions
of Theorems~\ref{thstr} and~\ref{thinvr} remain valid for the operators
$\mathcal{K}_{\sigma,r}f = \mathcal{G}_{\sigma}(A_{\gamma/\sigma,r}f)$.
\end{remark}

\subsection{Strong converse inequalities}
By employing arguments similar to those in the proofs of Theorem~\ref{thC} and Proposition~\ref{propdopconv}, we establish the following strong converse theorems for Kantorovich-type operators on $X(\mathbb{R})$.

\begin{theorem}\label{thCr}
  Let $X(\R)$ be a Banach lattice satisfying Assumption~A, and let $r, s \in \N$ with $2r\ge s$. Assume that $\mathcal{K}_{\s,r}$ satisfies assumptions~\eqref{a1r}, \eqref{a3r}, and~\eqref{a4r}.
  Then, for all $f\in X$, we have
  \begin{equation*}
    \|f-\mathcal{K}_{\s,r} f\|_{X}\asymp \|(I-\dot A_{\g/\s})^s f\|_{X},
  \end{equation*}
  where $\asymp$ denotes a two-sided inequality with positive constants independent of $f$ and $\s$.
\end{theorem}

\begin{proposition}\label{propdopconvr}
    Let $X, r, s$ be as in Theorem~\ref{thCr} with $r=2s$. Assume that  $\mathcal{K}_{\s,r}$ satisfies \eqref{a1r}, \eqref{a3r}, and
    \begin{equation*}
      \mathcal{K}_{\s,r}\vp=A_{\g/\s,r}\vp\quad\text{for every } \vp\in \Phi_\s.
    \end{equation*}
     Then for all $f\in X$, we have
  \begin{equation*}
    \|f-\mathcal{K}_{\s,r} f\|_{X}\asymp\|(I-A_{\g/\s})^r f\|_{X},
  \end{equation*}
where $\asymp$ denotes a two-sided inequality with positive constants independent of $f$ and $\s$.
\end{proposition}

\subsection{Examples} Similarly to Section~\ref{secext}, we discuss several classical examples of operators on $\R$, for which the above results can be applied.

We denote
\begin{equation}\label{Qr}
      \mathcal{K}_{\s,r}^\vp f(x):=\sum_{k\in\Z}A_{\g/\s,r}f\(\s^{-1}k\)V^\vp\(\s x-k\),
\end{equation}
where $V^\vp (x):=\widehat{\vp}(x)$ and $\vp$ is an even, bounded function with a compact support.

In what follows, to simplify the notation, we denote
$$
\mathcal{K}_n^\vp f := \mathcal{K}_{n,1}^\vp f
$$
and
$$
V_\s^\vp(x):=\s V^\vp (\s x).
$$

For our purposes, we employ the sampling operators
\begin{equation}\label{QrG}
  \mathcal{G}_{\sigma}^\varphi f(x) := \sum_{k\in \mathbb{Z}} f\left(\sigma^{-1}k\right) V^\varphi(\sigma x - k)
\end{equation}
together with the identity
\begin{equation}\label{eqKGr}
  \mathcal{K}_{\sigma,r}^\varphi f(x) = \mathcal{G}_\sigma^\varphi (A_{\gamma/\sigma,r}f)(x).
\end{equation}

Assume that $X$ satisfies Assumption~A. Since $V_\sigma^\varphi \in \mathcal{B}_X^{\lambda\sigma}$
for some $\lambda>0$, it follows that $V_\sigma^\varphi (\cdot-t) \in X$ for all $t \in \mathbb{R}$.
Indeed, by applying the Taylor formula and the Bernstein-type inequality
from Lemma~\ref{lemm3}(i), we obtain
\begin{equation*}
  \|V_\sigma^\varphi (\cdot-t)\|_X
  \le \sum_{\nu=0}^\infty \frac{\|(V_\sigma^\varphi)^{(\nu)}\|_X}{\nu!} |t|^\nu
  \le e^{C\lambda\sigma |t|} \|V_\sigma^\varphi\|_X < \infty.
\end{equation*}
In particular, this ensures that for every $g \in X'$, the convolution $g * V_\sigma^\varphi$
is well-defined in the classical sense.

We now prove several propositions that can be used to verify assumption~\eqref{a1r}--\eqref{a4r} for the operators~$\mathcal{K}_{\s,r}^\vp$.

\begin{proposition}\label{prex1r}
   Let $X(\R)$ be a Banach lattice satisfying Assumption~A such that
   \begin{equation}\label{zv1r}
     \sup_{\s\ge 1}\sup_{\|g\|_{X'}\le 1}\|g*V_\s^\vp\|_{X'}<\infty.
   \end{equation}
   Then

   {\rm (i)} for all $f\in X_\s$, the series in~\eqref{QrG} converges in $X$, and the operator $\mathcal{G}_\s^\vp$ satisfies~\eqref{b1r}.

   {\rm (ii)} for all $f\in X$ and $r\in \N$, the series in~\eqref{Qr} converges in $X$, and the operator $\mathcal{K}_{\s,r}^\vp$ satisfies~\eqref{a1r}.

   In particular, \eqref{zv1r} holds if $\widehat{\vp}\in \mathcal{R}^*$.
\end{proposition}

\begin{proof}
Statement (i), as well as the sufficiency of the condition $\widehat{\varphi} \in \mathcal{R}^*$
for the fulfillment of~\eqref{zv1r}, was established in~\cite[Proposition~4.1]{K26}.

To establish statement (ii), we note that the convergence of the series in~\eqref{Qr} in the norm of $X$
follows from statement (i) combined with the fact that $A_{\gamma/\s,r}f \in X_\sigma$
by virtue of~\eqref{Asfx}. Furthermore, utilizing the identity~\eqref{eqKGr},
Lemma~\ref{prex1}, and again~\eqref{Asfx}, we obtain
\begin{equation*}
  \|\mathcal{K}_{\sigma,r}^\varphi f\|_X
  = \|\mathcal{G}_{\sigma}^\varphi A_{\gamma/\sigma,r}f\|_X
  \le C \|A_{\gamma/\sigma,r}f\|_{X_\sigma}
  \le C \|f\|_X.
\end{equation*}
\end{proof}

\begin{lemma}\label{prex1r+}
  Let $g\in \mathcal{B}_{X}^{(2\pi-a)\s}$, $V_\s^\vp \in \mathcal{B}_{X'}^{a\s}$ for some $a\in (0,2\pi)$, and $r\in \N$. Then
  \begin{equation}\label{po}
    \mathcal{K}_{\s,r}^\vp g=A_{\g/\s,r}g*V_\s^\vp.
  \end{equation}
\end{lemma}

\begin{proof}
 Using Assumption~A, it is easy to see that $A_{\g/\s,r}g\in \mathcal{B}_{X}^{(2\pi-a)\s}$. Hence,  since $A_{\g/\s,r}g\, V_\s^\vp(\cdot-t)\in L_1(\R)$ for all $t\in \R$, and since $\supp \mathcal{F}(A_{\g/\s,r}g\, V_\s^\vp)\subset [-\pi\s,\pi\s]$, see, e.g.,~\cite[Theorem~6.37]{R73}, the Poisson summation formula yields~\eqref{po}.
\end{proof}

We now provide an analogue of Proposition~\ref{prex2}.
Recall that $\eta$ denotes a function in $C^\infty(\R)\cap \mathcal{R}$ such that $\eta(\xi)=1$ for $|\xi|\le 1$ and $\eta(\xi)=0$ for $|\xi|\ge 2$.

\begin{proposition}\label{prex2r}
   Let $X(\R)$ be a Banach lattice satisfying Assumption A and let
   $$
   w(\xi)=\frac{1-\vp(\xi)}{\xi^s}\eta(\xi),
   $$
  where $\supp \vp\subset [-2\pi+1,2\pi-1]$ and $w$ is understood to be continuously extended to $\xi=0$.
If $\widehat{w}\in\mathcal R^*$, then the operators $\mathcal{K}_{\s,r}^\vp$ ($r\in \N$ with $2r\ge s$) and $\mathcal{G}_\s^\vp$ satisfy~\eqref{a3r} and~\eqref{b2r}, respectively.
\end{proposition}

\begin{proof}
The assertion of the proposition for the general sampling operator $\mathcal{G}_s^\varphi$
was established in~\cite[Proposition~4.3(i)]{K26}. Consequently, the statement
for the Kantorovich-type operator $\mathcal{K}_{\sigma,r}^\varphi$ follows directly
from Remark~\ref{remGnr}.
\end{proof}

\begin{proposition}\label{prex2pr2r}
   Let $X$ be a Banach lattice satisfying Assumption A and let
     \begin{equation*}
     v(\xi)=\frac{\xi^s\eta(\xi)}{1-\vp(\xi)\(1-\(1-\sinc(\g\xi)\)^r\)},
    \end{equation*}
    where $\supp \vp\subset [-2\pi+1,2\pi-1]$ and $v$ is understood to be continuously extended to $\xi=0$.
    If $\widehat{v}\in\mathcal R^*$, then the operator $\mathcal{K}_{\s,r}^\vp$ ($r\in\N$ with $2r\ge s$) satisfies~\eqref{a4r}.
\end{proposition}

\begin{proof}
  Applying Lemma~\ref{prex1r+} to $g\in\mathcal{B}_X^\s$, we obtain the identity
   \begin{equation*}
     g^{(s)} = (g - \mathcal{K}_{\s,r}^\vp g)*V_\s^v.
   \end{equation*}
     Using this relation together with the principle of comparison for Fourier multipliers (see, e.g.,~\cite[Ch.~7]{TB} or \cite[Sec.~2.4]{V23}), we conclude that  inequality~\eqref{a4r} holds for $\mathcal{K}_{\s,r}^\vp$  if and only if there exists a constant $C>0$ such that
  \begin{equation*}
    \sup_{\s\ge 1}\|g*V_\s^v\|_X\le C\|g\|_X,\quad g\in\mathcal{B}_X^\s.
  \end{equation*}
    The latter inequality follows from  Lemma~\ref{lemS} since
  \begin{equation*}
  \begin{split}
        \|g*V_\s^v\|_{X}\le C\s \|\widehat{v}^*(\s\cdot)\|_{L_1(\R)}\|g\|_{X}\le C\|g\|_{X}.
  \end{split}
  \end{equation*}
\end{proof}

\begin{example}\label{ex1r}
Consider the Kantorovh-type operator $\mathcal{K}_\s^{\vp_1}$ generated by $\vp_1=\chi_{[-\pi\s,\pi\s]}$. This operator can be expressed via the classical Whittaker--Kotelnikov--Shannon sampling expansion by
$$
\mathcal{K}_\s^{\vp_1} f(x)=\sum_{k\in\Z}A_{\g/\s}f\(\s^{-1}k\){\rm sinc}\(2\pi(\s x-k)\).
$$

We denote
$$
\mathcal{D}_\sigma f(x) := \mathcal F^{-1}\big(\chi_{[-\pi\sigma,\pi\sigma]}(\xi)\,\widehat f(\xi)\big)(x),\quad f\in \mathcal{S}.
$$

\begin{proposition}\label{cor2rs}
  Let $X(\R)$ be a Banach lattice satisfying Assumption~A with ${\rm sinc}\in X$ and
  \begin{equation}\label{fDnr}
    \sup_{\sigma\ge1}\|\mathcal{D}_\sigma\|_{X\to X}<\infty.
  \end{equation}
Then for all $f\in X$, we have
  \begin{equation*}
    \|f-\mathcal{K}_{\s}^{\vp_1} f\|_{X}\asymp\|(I-A_{\g/\s}) f\|_{X},
  \end{equation*}
  where $\asymp$ denotes a two-sided inequality with positive constants independent of $f$ and $\s$.
\end{proposition}

\begin{proof}
By duality, we have $\sup_{\sigma\ge1}\|\mathcal{D}_\sigma\|_{X'\to X'}<\infty$. Moreover, $\mathcal{D}_\s f=f*{\rm sinc}_\s$ for all $f\in X'$. Together with Proposition~\ref{prex1r}, this implies that for every $f\in X$ the series $\mathcal{K}_\s^{\vp_1} f$ converges in $X$, and  $\mathcal{K}_\s^{\vp_1}$ satisfies assumption~\eqref{a1r}.
The rest of the proof is analogous to the proof of Proposition~\ref{cor2+} and is based on the well-known equality $\mathcal{G}_\s^{\vp_1} g=g$ for all $g\in \mathcal{B}_X^\s$, Proposition~\ref{propdopconvr}, and Lemma~\ref{lemm1}(iii).
\end{proof}
\end{example}

\begin{example}\label{ex2r}
Recall that
$$
\rho(\xi)=(1-\xi^2)_+^\a,\quad \a>0.
$$

\begin{proposition}\label{brr}
  Let $X(\R)$ be a Banach lattice satisfying Assumption A. Then for all $f\in X(\R)$ we have
  \begin{equation*}
    \|f-\mathcal{K}_{\s}^\rho f\|_{X}\asymp \|(I-A_{\g/\s}) f\|_{X},
  \end{equation*}
where $\asymp$ denotes a two-sided inequality with positive constants independent of $f$ and $\s$.
\end{proposition}

\begin{proof}
The proof is analogous to the proof of Proposition~\ref{br} and is based on Theorem~\ref{thCr} together with Propositions~\ref{prex2r} and~\ref{prex2pr2r}.
\end{proof}

We note that in the classical case $X=L_p(\mathbb{R})$, $1\le p<\infty$, analogues of Propositions~\ref{cor2rs} and~\ref{brr} were obtained in~\cite{KS21b}.
\end{example}

\begin{example}\label{ex1.5r}
Similarly as in the periodic case (see Example~\ref{ex1.5}), we have that the Kantorovich operators~\eqref{Qr}
generated by
$$
\vp_2(\xi)=\frac{\eta(2\xi)}{\sinc(\g\xi)}\quad\text{and}\quad \vp_3(\xi)=\frac{\chi_{[-1,1]}(\xi)}{\sinc(\g\xi)} ,\quad 0<\g\le\frac\pi2,
$$
provide the same order of convergence as the error of best approximation by functions from~$\mathcal{B}_X^\s$. Namely, repeating the proofs of Propositions~\ref{prKEas} and~\ref{prKEas2} and applying herewith Proposition~\ref{prex1r} and Lemma~\ref{prex1r+}, we obtain the following results.

\begin{proposition}\label{prKEasr}
  Let $X(\R)$ be a Banach lattice satisfying Assumption~A. Then for all $f\in X$, we have
  \begin{equation}\label{eqKEr}
    E_\s(f)_X\le \|f-\mathcal{K}_{\s}^{\vp_2} f\|_X\le CE_{\s/2}(f)_X,
  \end{equation}
  where the constat $C$ is independent of $f$ and $\s$.
\end{proposition}

Under an additional restriction on $X(\R)$, we have more shaper version of relation~\eqref{eqKEr}.

\begin{proposition}\label{prKEasr2}
  Let $X(\R)$ be a Banach lattice satisfying Assumption~A and~\eqref{fDnr}. Then for all $f\in X$, we have
  \begin{equation*}
    \|f-\mathcal{K}_\s^{\vp_3} f\|_X\asymp E_\s(f)_X,
  \end{equation*}
where $\asymp$ is a two-sided inequality with positive constants independent of $f$ and $\s$.
\end{proposition}
\end{example}

\end{document}